\documentclass[a4paper,12pt]{amsart}

\usepackage{a4wide}
\usepackage{tikz}

\newif\ifdetails
\detailstrue
\newcommand{\DETAIL}[1]%
{\ifdetails\par\fbox{\begin{minipage}{0.9\linewidth}\textit{Detail:}
      #1\end{minipage}}\par\fi}
\newcommand{\TODO}[1]%
{\ifdetails\par\fbox{\begin{minipage}{0.9\linewidth}\textbf{TODO:}
      #1\end{minipage}}\par\fi}

\usepackage{makecell}
\usepackage{relsize}

\newtheorem{lemma}{Lemma}

\newtheorem{theorem}[lemma]{Theorem}
\newtheorem{corollary}[lemma]{Corollary}
\theoremstyle{remark}

\newtheorem{remark}{Remark}

\newtheorem{question}{Question}

\usepackage{caption}
\usepackage{subcaption}
\usepackage{float} % provides H as float placement specifier
\usepackage[pdftex,a4paper,
citecolor = blue, colorlinks=true,urlcolor=blue]{hyperref}
\newcommand*{\KeepStyleUnderBrace}[1]{%
  \mathop{%
    \mathchoice
    {\underbrace{\displaystyle#1}}%
    {\underbrace{\textstyle#1}}%
    {\underbrace{\scriptstyle#1}}%
    {\underbrace{\scriptscriptstyle#1}}%
  }\limits
}

\newcommand{\old}[1]{{}}
\title{The Minimum asymptotic density of binary caterpillars}

\author{Audace A. V. Dossou-Olory}
\thanks{The author was supported by Stellenbosch University and African Institute for Mathematical Science (AIMS) South Africa.}
\address{Audace A. V. Dossou-Olory\\ Department of Mathematical Sciences \\ Stellenbosch University \\ Private Bag X1, Matieland 7602 \\ South Africa}
\email{audaced@sun.ac.za}
\subjclass[2010]{Primary 05C05; secondary 05C07, 05C30, 05C35}
\keywords{caterpillars, minimum asymptotic density, leaf-induced subtrees, $d$-ary trees, inducibility, complete $d$-ary trees, strictly $d$-ary trees.}

\begin{document}

\begin{abstract}
Given $d\geq 2$ and two rooted $d$-ary trees $D$ and $T$ such that $D$ has $k$ leaves, the density $\gamma(D,T)$ of $D$ in $T$ is the proportion of all $k$-element subsets of leaves of $T$ that induce a tree isomorphic to $D$, after erasing all vertices of outdegree $1$. In a recent work, it was proved that the limit inferior of this density as the size of $T$ grows to infinity is always zero unless $D$ is the $k$-leaf binary caterpillar $F^2_k$ (the binary tree with the property that a path remains upon removal of all the $k$ leaves). Our main theorem in this paper is an exact formula (involving both $d$ and $k$) for the limit inferior of $\gamma(F^2_k,T)$ as the size of $T$ tends to infinity.
\end{abstract}

\maketitle

\section{Preliminaries and statement of the main results}

Throughout the whole note, $d$ will always denote a fixed positive integer greater than $1$. A rooted tree is called a \textit{$d$-ary tree} if each of its non-leaf vertices has outdegree at most $d$ but at least $2$. In the case $d=2$, we shall simply speak of \textit{binary} trees, and in the case $d=3$, we shall speak of \textit{ternary} trees.

In a recent paper~\cite{AudaceStephanPaper1}, Czabarka, Sz\'{e}kely, Wagner and the author of the current note investigated the \textit{inducibility} of $d$-ary trees, which can be thought of as the \textit{maximum asymptotic density} of a $d$-ary tree occurring as a subtree induced by leaves of another $d$-ary tree with sufficiently large number of leaves (a formal definition will be given later in this section).

For two $d$-ary trees $D$ and $T$ such that $D$ has $k$ leaves, we shall denote by $\gamma(D,T)$ the proportion of all $k$-element subsets of leaves of $T$ that induce a tree isomorphic (in the sense of rooted trees) to $D$, after erasing all vertices that have outdegree $1$. The tree induced by a subset $L$ of leaves of the leaf-set of a $d$-ary tree $T$ is obtained by first taking the minimal subtree containing all the leaves in $L$, and then erasing vertices that have outdegree $1$. Every tree obtained in this manner will be referred to as a \textit{leaf-induced subtree} of $T$; see Figure~\ref{leaf-induced} for an illustration.

\begin{figure}[!h]\centering  
\begin{tikzpicture}[thick]
\node [circle,draw] (r) at (0,0) {};

\draw (r) -- (-2,-2);
\draw (r) -- (0,-4);
\draw (r) -- (2,-2);
\draw (-2,-2) -- (-2.5,-4);
\draw (-2,-2) -- (-1.5,-4);
\draw (2,-2) -- (1,-4);
\draw (2,-2) -- (2,-4);
\draw (2,-2) -- (3,-4);
\draw (1.25,-3.5) -- (1.5,-4);

\node [fill,circle, inner sep = 2pt ] at (-2.5,-4) {};
\node [fill,circle, inner sep = 2pt ] at (-1.5,-4) {};
\node [fill,circle, inner sep = 2pt ] at (0,-4) {};
\node [fill,circle, inner sep = 2pt ] at (1,-4) {};
\node [fill,circle, inner sep = 2pt ] at (1.5,-4) {};
\node [fill,circle, inner sep = 2pt ] at (2,-4) {};
\node [fill,circle, inner sep = 2pt ] at (3,-4) {};

\node at (-1.5,-4.5) {$\ell_1$};
\node at (0,-4.5) {$\ell_2$};
\node at (1,-4.5) {$\ell_3$};
\node at (2,-4.5) {$\ell_4$};

\node [circle,draw] (r1) at (6,-2) {};

\draw (r1) -- (5,-4);
\draw (r1) -- (6,-4);
\draw (r1) -- (7,-4);
\draw (6.75,-3.5)--(6.5,-4);

\node [fill,circle, inner sep = 2pt ] at (5,-4) {};
\node [fill,circle, inner sep = 2pt ] at (6,-4) {};
\node [fill,circle, inner sep = 2pt ] at (6.5,-4) {};
\node [fill,circle, inner sep = 2pt ] at (7,-4) {};

\node at (5,-4.5) {$\ell_1$};
\node at (6,-4.5) {$\ell_2$};
\node at (6.5,-4.5) {$\ell_3$};
\node at (7,-4.5) {$\ell_4$};

\end{tikzpicture}
\caption{A ternary tree and the subtree induced by the four leaves $\ell_1,\ell_2,\ell_3,\ell_4$.}\label{leaf-induced}
\end{figure}
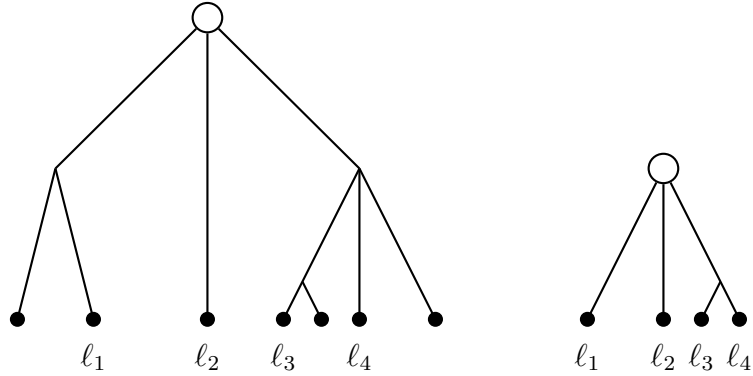

\medskip
By a \textit{copy} of $D$ in $T$, we mean any leaf-induced subtree of $T$ which is isomorphic to $D$. We shall denote by $c(D,T)$ the total number of copies of $D$ in $T$ and by $|T|$ the number of leaves of $T$. So $\gamma(D,T)$ is the ratio
\begin{align*}
\frac{c(D,T)}{\dbinom{|T|}{|D|}}
\end{align*}
by definition. For brevity, $\gamma(D,T)$ will be called the \textit{density} of $D$ in $T$.

\medskip
In~\cite{AudaceStephanPaper1}, the inducibility of a $d$-ary tree is defined as being the maximum asymptotic density of $D$. Formally speaking, the inducibility $I_d(D)$ of a $d$-ary tree $D$ is the limit superior of the density of $D$ in $T$ as the number of leaves of $T$ grows to infinity. One of the principal results in \cite{AudaceStephanPaper1} is that
\begin{align}\label{IndDA}
I_d(D)=\limsup_{\substack{|T| \to \infty\\ T~\text{$d$-ary tree}}} \gamma(D,T)=\lim_{n\to \infty} \max_{\substack{|T|=n \\ T~\text{$d$-ary tree}}} \gamma(D,T)\,.
\end{align}

For our purposes, let us define and call the quantity
\begin{align*}
\liminf_{\substack{|T| \to \infty\\ T~\text{$d$-ary tree}}} \gamma(D,T)
\end{align*}
the \textit{minimum asymptotic density} of the $d$-ary tree $D$ in $d$-ary trees. To put it another way, we mean
\begin{align*}
\liminf_{\substack{|T| \to \infty\\ T~\text{$d$-ary tree}}} \gamma(D,T)=\lim_{n\to \infty} \min_{\substack{|T|=n \\ T~\text{$d$-ary tree}}} \gamma(D,T)\,,
\end{align*}
where the proof of existence of the limit is analogous to that in~\eqref{IndDA}.

\medskip
An important and recurring theme that appears throughout extremal graph theory is finding the minimum or maximum value of a given graph invariant within a class of graphs all sharing a certain property. Understanding an invariant provides information about the structure of a graph. In particular, the problem of characterising the extremal graphs has been and continues to be a topic of a great interest to graph theorists.

It is therefore natural to consider the problem of determining the minimum asymptotic density of a $d$-ary tree $D$ in $d$-ary trees. Quite fascinatingly however, it turns out that the study of this problem reduces to the study of the minimum asymptotic density of so-called binary caterpillars. 

\medskip
A \textit{binary caterpillar} is a binary tree with the property that its non-leaf vertices form a path starting at the root. We shall denote by $F^2_k$ the binary caterpillar with $k$ leaves -- see Figure~\ref{binCatF4} for the binary caterpillar with four leaves.

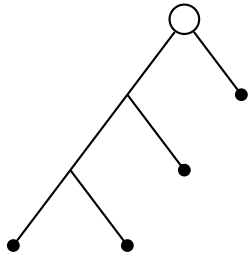
\begin{figure}[htbp]\centering
\begin{tikzpicture}[thick,level distance=10mm]
\tikzstyle{level 1}=[sibling distance=15mm]
\node [circle,draw]{}
child {child {child {[fill] circle (2pt)}child {[fill] circle (2pt)}}child {[fill] circle (2pt)}}
child {[fill] circle (2pt)};
\end{tikzpicture}
\caption{The $4$-leaf binary caterpillar $F^2_4$.}\label{binCatF4}
\end{figure}

The following fundamental result from~\cite{AudaceStephanPaper1} characterises all the $d$-ary trees with the maximal inducibility:

\begin{theorem}[\cite{AudaceStephanPaper1}]\label{max1dAry}
Let $d\geq 2$ be an arbitrary but fixed positive integer. Among $d$-ary trees, only binary caterpillars have inducibility $1$.
\end{theorem}

It follows immediately from Theorem~\ref{max1dAry} that
\begin{align*}
\liminf_{\substack{|T|\to \infty \\T~\text{$d$-ary tree}}} \gamma(D,T) =0
\end{align*}
for every $d$, as soon as $D$ is not a binary caterpillar because
\begin{align*}
0\leq \liminf_{\substack{|T|\to \infty \\T~\text{$d$-ary tree}}} \gamma(D,T) \leq \liminf_{\substack{|T|\to \infty \\T~\text{$d$-ary tree}}} \big(1-\gamma(F^2_{|D|},T)\big)=1-I_d\big(F^2_{|D|}\big)\,,
\end{align*}
provided that $D$ is not isomorphic to $F^2_{|D|}$. However, at this point, it is not clear a priori that
\begin{align*}
\liminf_{\substack{|T|\to \infty \\T~\text{$d$-ary tree}}} \gamma(F^2_k,T)  >0
\end{align*}
for every $k$---one will have to put more effort in finding out what the minimum asymptotic density of binary caterpillars might be for every $d$. Thus, the problem we address in this note can be formulated as follows:

\medskip
\textbf{Problem}: \textit{Given a binary caterpillar $F^2_k$, is it true that every $d$-ary tree $T$ with sufficiently large number of leaves always contains a positive density of $F^2_k$? If so, what is the asymptotic minimum number of copies of $F^2_k$ in a $d$-ary tree with large enough number of leaves?}

\medskip
Let us mention that binary caterpillars have been proved to be extremal among binary trees with respect to some other graph parameters. For instance, binary caterpillars have been shown in~\cite{fischermann2002wiener} to have the maximum Wiener index (sum of distances between all unordered pairs of vertices) among all binary trees with a prescribed number of leaves. In~\cite{szekely2005subtrees}, Sz{\'e}kely and Wang proved that binary caterpillars minimise the number of subtrees among all binary trees with a given number of leaves.

\medskip
In the following, we shall prove that the minimum asymptotic density of an arbitrary binary caterpillar is strictly positive for every $k$. In fact, we shall even be able to derive the precise value of this limiting quantity for every $k$. Clearly,
\begin{align*}
\liminf_{\substack{|T|\to \infty \\T~\text{$d$-ary tree}}} \gamma(F^2_k,T) =1
\end{align*}
for $k\leq 2$. Let us now proceed to find its value as a function of $d$ and $k$ for $k\geq 3$.

\medskip
A $d$-ary tree will be called a \textit{strictly $d$-ary tree} if each of its vertices has outdegree $0$ or $d$. Our main result reads as follows:

\begin{theorem}\label{minVslimInf}
Let $d,k\geq 2$ be arbitrary but fixed positive integers. Then the following double identity
\begin{align*}
\liminf_{\substack{|T|\to \infty \\T~\text{$d$-ary tree}}} \gamma(F^2_k,T) =\liminf_{\substack{|T|\to \infty \\ T~\text{strictly $d$-ary tree}}}  \gamma(F^2_k,T)=\frac{k!}{2} \cdot (d-1)^{k-1} \cdot \prod_{j=1}^{k-1}(d^j-1)^{-1}
\end{align*}
holds. Furthermore, we have
\begin{align*}
\min_{\substack{|T|=n\\T~\text{$d$-ary tree}}} \gamma\big(F^2_k,T \big) \leq \liminf_{\substack{|T|\to \infty \\T~\text{$d$-ary tree}}} \gamma(F^2_k,T)
\end{align*}
for every $k$ and $n\geq k$. 
\end{theorem}

Formally, the first equality in Theorem~\ref{minVslimInf} tells us that the minimum asymptotic density of any binary caterpillar can also be computed by restricting the set of $d$-ary trees over which the minimum is taken to strictly $d$-ary trees only. This situation, in a certain sense, parallels the opposite problem concerning the maximum asymptotic density $I_d(D)$ of a $d$-ary tree $D$, where the authors of paper~\cite{AudaceStephanPaper1} could prove that $I_d(D)$ satisfies the equivalent identity
\begin{align*}
I_d(D)=\lim_{n\to \infty} \max_{\substack{|T|=n \\ T~\text{strictly $d$-ary tree}}} \gamma(D,T)
\end{align*}
for every $d$-ary tree $D$. So, it may also be immediately clear that the identity
\begin{align*}
\liminf_{\substack{|T|\to \infty \\T~\text{$d$-ary tree}}} \gamma(F^2_k,T) =\liminf_{\substack{|T|\to \infty \\ T~\text{strictly $d$-ary tree}}}  \gamma(F^2_k,T)
\end{align*}
holds for every $k$. Indeed, it is shown in~\cite{AudaceStephanPaper1} that for every $d$-ary tree $T$ with sufficiently large number of leaves, there exists a strictly $d$-ary tree $T^*$ such that $|T^*|\geq |T|$ and the asymptotic formula
\begin{align*}
\gamma\big(D,T\big)=\gamma\big(D,T^*\big) +\mathcal{O}\big(|T|^{-1} \big)\,.
\end{align*}
holds for every $d$-ary tree $D$, where the $\mathcal{O}$-constant depends on $d$ only (and nothing else!). 

\begin{remark}
The special cases $d=2$ and $d=3$ of Theorem~\ref{minVslimInf} correspond to
\begin{align*}
\liminf_{\substack{|T| \to \infty \\ T~\text{binary tree}}}\gamma\big(F^2_3,T \big)=1~~\text{and}~~
\liminf_{\substack{|T| \to \infty \\ T~\text{ternary tree}}}\gamma\big(F^2_3,T \big)=\frac{3}{4}\,,
\end{align*}
respectively. In particular, it displays the following equivalence in ternary trees:
\begin{align*}
\liminf_{\substack{|T|\to \infty \\ T~\text{ternary tree}}}\gamma\big(F^2_3,T \big)= 1-I_3(C_3)
\end{align*}
as the star $C_3$ (consisting of a root and three leaves attached to it) and the binary caterpillar $F^2_3$ are the only $3$-leaf $d$-ary trees for every $d> 2$. Moreover, this confirms that the inducibility of $C_3$ in ternary trees is $1/4$---see~\cite[Theorem~1]{AudaceStephanPaper1}. 
\end{remark}

The next corollary will follow from the proof of Theorem~\ref{minVslimInf}:

\begin{corollary}\label{cor:Theo}
Let $d,k\geq 2$ be arbitrary but fixed positive integers. Then the minimum number of copies of the binary caterpillar $F^2_k$ in an arbitrary $n$-leaf $d$-ary tree $T$ is asymptotically 
\begin{align*}
\frac{ n^{k}}{2} \cdot (d-1)^{k-1} \cdot \prod_{j=1}^{k-1}(d^j-1)^{-1} + \mathcal{O} (n^{k-1})
\end{align*}
as $n \to \infty$.
\end{corollary}

We should mention that the fact that the binary caterpillar has positive minimum asymptotic density was actually the key result in~\cite{czabarka2016inducibility} for the application to the tanglegram crossing problem (see the proof of Lemma 11 in~\cite{czabarka2016inducibility}).

\section{Proof of the main theorem and its corollary}

This section carries a proof of Theorem~\ref{minVslimInf} as well as a proof of Corollary~\ref{cor:Theo}. But before we get to the proofs of these results, we need to go through some preparation.

\medskip
Rooted trees are predestined for recursive approaches. For a $d$-ary tree $D$ with branches $D_1,D_2,\ldots,D_r$, we define the equivalence relation $\sim_{D}$ on the set of all permutations of the indices $1,2,\ldots,r$ as follows: for two permutations $\pi$ and $\pi^{\prime}$ of $\{1,2,\ldots,r\}$,
\begin{align*}
\big(\pi(1),\pi(2),\ldots,\pi(r)\big) \sim_{D} \big(\pi^{\prime}(1),\pi^{\prime}(2),\ldots,\pi^{\prime}(r)\big)
\end{align*}
if for every $j \in \{1,2,\ldots,r\}$, the tree $D_{\pi(j)}$ is isomorphic (in the sense of rooted trees) to the tree $D_{\pi^{\prime}(j)}$. 

Further, we denote by $M(D)$ a complete set of representatives of all equivalence classes of the equivalence relation $\sim_{D}$. Thus, if $m_1,m_2,\ldots,m_c$ denote the multiplicities of the branches of $D$ with respect to isomorphism, then the size of $M(D)$ is exactly
\begin{align*}
|M(D)|= \binom{r}{m_1,m_2,\ldots,m_c}\,.
\end{align*}

Bearing this notation in mind, we get the following recursion
\begin{align}\label{The general recursion}
c(D,T)=\sum_{i=1}^d c(D,T_i)+\sum_{\substack{\{i_1,i_2,\ldots,i_r\}\subseteq \{1,2,\ldots,d\}}}~~\sum_{\pi \in M(D)}~\prod_{j=1}^r c\big(D_{\pi(j)},T_{i_j}\big)\,,
\end{align}
which is valid for every strictly $d$-ary tree $T$ with branches $T_1,T_2,\ldots,T_d$. The proof of this formula is straightforward. In words, this formula is established as follows:
\begin{itemize}
\item The term $\sum_{i=1}^d c(D,T_i)$ corresponds to the sum, over all branches of $T$, of the total number of subsets of leaves in a single branch of $T$ that induce a copy of $D$.
\item The expression $\prod_{j=1}^r c\big(D_{\pi(j)},T_{i_j}\big)$ stands for the number of copies of $D$ in which its branches $D_{\pi(1)},D_{\pi(2)},\ldots, D_{\pi(r)}$ are induced by subsets of leaves of $T_{i_1},T_{i_2}, \ldots,T_{i_r}$, respectively. We run this product for every subset of $r$ elements of the set of branches of $T$ and for every permutation $\pi$ in $M(D)$, as to take into consideration the possibility that some branches of $D$ might be isomorphic.
\end{itemize}

\medskip
The \textit{complete $d$-ary tree of height $h$} is the strictly $d$-ary tree in which all the leaves reside at the same distance $h$ from the root. We shall denote it by $CD^d_h$. Note that $CD^d_h$ has $d^h$ leaves in total.

The complete $4$-ary tree of height $2$ is shown in Figure~\ref{CD4.2}.

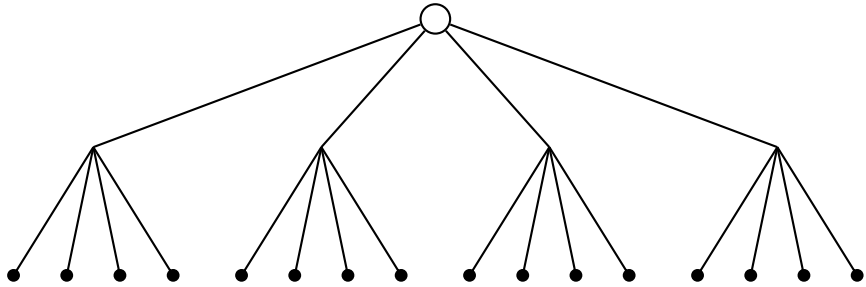
\begin{figure}[htbp]\centering
\begin{tikzpicture}[thick,level distance=17mm]
\tikzstyle{level 1}=[sibling distance=30mm]
\tikzstyle{level 2}=[sibling distance=7mm]
\node [circle,draw]{}
child {child {[fill] circle (2pt)}child {[fill] circle (2pt)}child {[fill] circle (2pt)}child {[fill] circle (2pt)}}
child {child {[fill] circle (2pt)}child {[fill] circle (2pt)}child {[fill] circle (2pt)}child {[fill] circle (2pt)}}
child {child {[fill] circle (2pt)}child {[fill] circle (2pt)}child {[fill] circle (2pt)}child {[fill] circle (2pt)}}
child {child {[fill] circle (2pt)}child {[fill] circle (2pt)}child {[fill] circle (2pt)}child {[fill] circle (2pt)}};
\end{tikzpicture}
\caption{The complete $4$-ary tree of height $2$.}\label{CD4.2}
\end{figure}

\medskip
The following formula can be found explicitly in the proof of Theorem~1 of paper~\cite{AudaceStephanPaper1}: for every fixed positive integer $d\geq 2$, we have
\begin{align}\label{StarForm}
c(CD^r_1,CD^d_h)=\frac{\dbinom{d}{r}}{d^r-d} \cdot \big(d^{r\cdot h} -d^h\big)
\end{align}
for every $r>1$ and all $h\geq 1$ (in \cite{AudaceStephanPaper1}, the tree $CD^r_1$ is called the $r$-leaf star).

\medskip
A \textit{$d$-ary caterpillar} is a strictly $d$-ary tree with the property that every non-leaf vertex has $d-1$
adjacent vertices that are leaves, except for the lowest which has $d$ adjacent vertices that are leaves. Note that the non-leaf vertices must lie on a single path. We shall denote the $d$-ary caterpillar with $k$ leaves by $F^d_k$ -- see Figure~\ref{Ternary caterpillars.} for ternary caterpillars.

\begin{figure}[htbp]\centering
%3    
  \begin{subfigure}[b]{0.2\textwidth} \centering  
\begin{tikzpicture}[thick,level distance=10mm]
\tikzstyle{level 1}=[sibling distance=10mm]
\node [circle,draw]{}
child {[fill] circle (2pt)}
child {[fill] circle (2pt)}
child {[fill] circle (2pt)};
\end{tikzpicture}
   \caption{$F^3_3$}
  \end{subfigure}\qquad
%5
\begin{subfigure}[b]{0.2\textwidth} \centering 
  \begin{tikzpicture}[thick,level distance=10mm]
\tikzstyle{level 1}=[sibling distance=10mm]
\tikzstyle{level 2}=[sibling distance=10mm]
\node [circle,draw]{}
child {[fill] circle (2pt)}
child {child {[fill] circle (2pt)}child {[fill] circle (2pt)}child {[fill] circle (2pt)}}
child {[fill] circle (2pt)};
\end{tikzpicture}
   \caption{$F^3_5$}
  \end{subfigure} \qquad
  %7
\begin{subfigure}[b]{0.2\textwidth} \centering 
  \begin{tikzpicture}[thick,level distance=9mm]
\tikzstyle{level 1}=[sibling distance=10mm]
\tikzstyle{level 2}=[sibling distance=10mm]
\node [circle,draw]{}
child {[fill] circle (2pt)}
child {child {[fill] circle (2pt)}child {child {[fill] circle (2pt)}child {[fill] circle (2pt)}child {[fill] circle (2pt)}}child {[fill] circle (2pt)}}
child {[fill] circle (2pt)};
\end{tikzpicture}
   \caption{$F^3_7$}
  \end{subfigure} 
	\caption{Ternary caterpillars $F^3_k$.} \label{Ternary caterpillars.}
\end{figure}
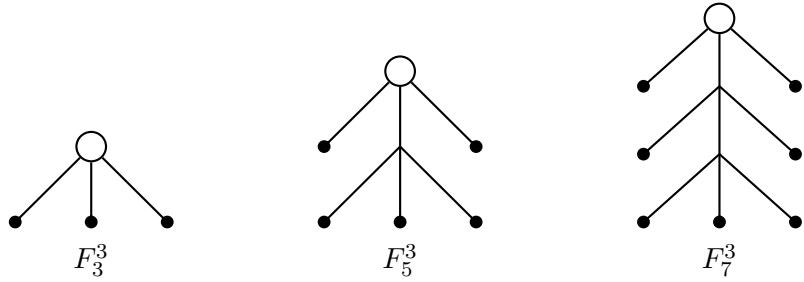

\medskip
In the following theorem, we derive an exact formula for the number of copies of the $r$-ary caterpillar with $k$ leaves in a complete $d$-ary tree of arbitrary height:

\begin{theorem}\label{CompProport}
Let an arbitrary positive integer $d\geq 2$ be fixed. For every $r \in \{2,3,\ldots,d\}$, the number of copies of $F^r_k$ in $CD^d_h$ is
\begin{align*}
c\big(F^r_k,CD^d_h\big)=\binom{d}{r}^{\frac{k-1}{r-1}}\cdot \Big(\frac{r}{d} \Big)^{\frac{k-r}{r-1}} \cdot d^{h-1} \cdot \prod_{i=1}^{\frac{k-1}{r-1}} \Bigg(\frac{d^{h\cdot (r-1)}-d^{(i-1)\cdot (r-1)}}{d^{i\cdot (r-1)}-1} \Bigg)
\end{align*}
for every $k>1$ and all $h\geq 1$. In particular, we have
\begin{align*}
\lim_{h\to \infty} \gamma(F^r_k,CD^d_h)=\frac{k!}{d} \cdot \dbinom{d}{r}^{\frac{k-1}{r-1}} \cdot  \Big(\frac{r}{d} \Big)^{\frac{k-r}{r-1}} \cdot \prod_{j=1}^{\frac{k-1}{r-1}}\big(d^{(r-1)j}-1 \big)^{-1}\,.
\end{align*}
\end{theorem}

\begin{proof}
For $k=r$, the formula of the theorem reads as 
\begin{align*}
c\big(F^r_k,CD^d_h\big)=\binom{d}{r} \cdot d^{h-1} \cdot \Bigg(\frac{d^{h\cdot (r-1)}-1}{d^{r-1}-1} \Bigg)=\binom{d}{r} \cdot \frac{d^{h\cdot r}- d^h}{d^r-d}\,, 
\end{align*}
and this agrees with equation~\eqref{StarForm}. For $h=1$, the formula of the theorem reads as
\begin{align*}
c\big(F^r_k,CD^d_h\big)=\binom{d}{r}^{\frac{k-1}{r-1}}\cdot \Big(\frac{r}{d} \Big)^{\frac{k-r}{r-1}} \cdot \prod_{i=1}^{\frac{k-1}{r-1}} \Bigg(\frac{d^{r-1}-d^{(i-1)\cdot (r-1)}}{d^{i\cdot (r-1)}-1} \Bigg)\,,
\end{align*}
which is equal to $0$ as soon as $k>r$, and $\binom{d}{r}$ when $k=r$. So this is again true because for $k>r$, it is clear that there cannot be any copies of $F^r_k$ in $CD^d_1$.

\medskip
Assume $k> r$ and $h> 1$. It is easy to see that the specialisation $T=CD^d_h$ and $D=F^r_k$ in equation~\eqref{The general recursion} yields the following recurrence relation:
\begin{align*}
c\big(F^r_k,CD^d_h\big)=d \cdot c\big(F^r_k,CD^d_{h-1}\big) + r\cdot \binom{d}{r}\cdot d^{(h-1)\cdot (r-1)} \cdot c\big(F^r_{k-r+1},CD^d_{h-1}\big)
\end{align*}
as all the $d$ branches of $CD^d_h$ are isomorphic to $CD^d_{h-1}$.

Making use of this recursion, we then continue the proof of the theorem by induction on $h$. Applying the induction hypothesis, we obtain
\begin{align*}
c\big(F^r_k,CD^d_h\big)&=\binom{d}{r}^{\frac{k-1}{r-1}}\cdot \Big(\frac{r}{d} \Big)^{\frac{k-r}{r-1}} \cdot d^{h-1} \cdot \Bigg[ \prod_{i=1}^{\frac{k-1}{r-1}} \Bigg(\frac{d^{(h-1)\cdot (r-1)}-d^{(i-1)\cdot (r-1)}}{d^{i\cdot (r-1)}-1} \Bigg)\\
& \hspace*{4cm} + d^{(h-1)\cdot (r-1)} \cdot \prod_{i=1}^{\frac{k-r}{r-1}} \Bigg(\frac{d^{(h-1)\cdot (r-1)}-d^{(i-1)\cdot (r-1)}}{d^{i\cdot (r-1)}-1} \Bigg) \Bigg]\,.
\end{align*}

We further manipulate this equation and we get
\begin{align*}
c\big(F^r_k,CD^d_h\big)&=\binom{d}{r}^{\frac{k-1}{r-1}}\cdot \Big(\frac{r}{d} \Big)^{\frac{k-r}{r-1}} \cdot d^{h-1} \cdot \Bigg[ d^{1-k}\cdot \prod_{i=1}^{\frac{k-1}{r-1}} \Bigg(\frac{d^{h\cdot (r-1)}-d^{i\cdot (r-1)}}{d^{i\cdot (r-1)}-1} \Bigg)\\
& \hspace*{4cm} + d^{r-k}\cdot d^{(h-1)\cdot (r-1)} \cdot \prod_{i=1}^{\frac{k-r}{r-1}} \Bigg(\frac{d^{h\cdot (r-1)}-d^{i\cdot (r-1)}}{d^{i\cdot (r-1)}-1} \Bigg) \Bigg]\\
&=\binom{d}{r}^{\frac{k-1}{r-1}}\cdot \Big(\frac{r}{d} \Big)^{\frac{k-r}{r-1}} \cdot d^{h-1}\\
&\cdot \Bigg[\frac{d^{h\cdot (r-1)} - d^{k-1}}{d^{k-1}\cdot (d^{h\cdot (r-1)}-1)}+\frac{d^{(h-1)\cdot(r-1)}(d^{k-1}-1)}{d^{k-r} \cdot (d^{h\cdot (r-1)}-1)}  \Bigg] \cdot
\prod_{i=1}^{\frac{k-1}{r-1}} \Bigg(\frac{d^{h\cdot (r-1)}-d^{(i-1)\cdot (r-1)}}{d^{i\cdot (r-1)}-1} \Bigg)\,,
\end{align*}
completing the induction step and thus the proof of the first part of the theorem.

\medskip
For the assertion on the limit, we note that
\begin{align*}
c\big(F^r_k,CD^d_h\big)&=\binom{d}{r}^{\frac{k-1}{r-1}}\cdot \Big(\frac{r}{d} \Big)^{\frac{k-r}{r-1}} \cdot d^{h-1} \cdot \Big(d^{h\cdot (k-1)} +\mathcal{O}\big(d^{h\cdot (k-r)} \big)\Big)\cdot \prod_{i=1}^{\frac{k-1}{r-1}} \big(d^{i\cdot (r-1)}-1\big)^{-1}\\
&=\binom{d}{r}^{\frac{k-1}{r-1}}\cdot \Big(\frac{r}{d} \Big)^{\frac{k-r}{r-1}} \cdot d^{-1} \cdot d^{k\cdot h} \cdot \prod_{i=1}^{\frac{k-1}{r-1}} \big(d^{i\cdot (r-1)}-1 \big)^{-1}  +\mathcal{O}\big(d^{h\cdot (k-r+1)} \big) \,,
\end{align*}
which implies that
\begin{align*}
\lim_{h\to \infty} \gamma\big(F^r_k,CD^d_h\big)=\frac{k!}{d} \cdot \dbinom{d}{r}^{\frac{k-1}{r-1}} \cdot  \Big(\frac{r}{d} \Big)^{\frac{k-r}{r-1}} \cdot \prod_{j=1}^{\frac{k-1}{r-1}}\big(d^{(r-1)j}-1\big)^{-1}
\end{align*}
as desired.
\end{proof}

\medskip
Our approach to the proof of Theorem~\ref{minVslimInf} consists of the following steps:
\begin{itemize}
\item First, we determine the density of $F^2_k$ in $CD^d_h$ as $h\to \infty$.
\item Next, we prove two auxiliary lemmas.
\item Employing the lemmas, we determine an explicit lower bound on $c(F^2_k,T)$ valid for all strictly $d$-ary trees $T$.
\item Finally, we mention that the bound on $\gamma(F^2_k,T)$ is achieved by complete $d$-ary trees in the limit.
\end{itemize}

\medskip
We replace $r$ with $2$ in the formula of $\lim_{h\to \infty} \gamma(F^r_k,CD^d_h)$ given in Theorem~\ref{CompProport} to obtain:

\begin{corollary}\label{FkrInCDh}
For the $k$-leaf binary caterpillar $F^2_k$, we have
\begin{align*}
\lim_{h\to \infty} \gamma(F^2_k,CD^d_h)= \frac{k!}{2} \cdot (d-1)^{k-1} \cdot \prod_{j=1}^{k-1}(d^j -1)^{-1}
\end{align*}
for every $d\geq 2$ and $k\geq 2$.
\end{corollary}

As a second step, we need two lemmas. Given positive integers $d\geq 2$ and $k\geq 3$, set
\begin{align*}
V_{d,k}=\Big\{(i_1,i_2,\ldots, i_d):~& i_1,i_2,\ldots,i_d~\text{nonnegative integers},\\
& i_1+i_2+\cdots +i_d=k,~\text{and none of them is $k$} \Big\}\,.
\end{align*}

\begin{lemma}\label{lemUseful}
For every given positive integer $k\geq 3$, we have
\begin{align*}
\sup_{\substack{0<x_1,x_2,\ldots,x_d<1\\x_1+x_2+\cdots + x_d=1}} \frac{\sum_{1\leq i<j\leq d} \Big(x_i\cdot x_j^{-1+k}+x_j\cdot x_i^{-1+k} \Big)}{1-\sum_{i=1}^d x_i^k} =\frac{1}{k}
\end{align*}
for every positive integer $d\geq 2$. 
\end{lemma}

\begin{proof}
Fix $d\geq 2$ and $k\geq 3$. Let $V^*_{d,k}$ denote the maximal subset of $V_{d,k}$ that contains no permutation of $\{1,k-1,\KeepStyleUnderBrace{0,0,\ldots,0}_{(d-2)~0's}\}$. Then we have the decomposition
\begin{align*}
1-\sum_{i=1}^d x_i^{k}&=\sum_{\substack{(i_1,i_2,\ldots,i_d)\in V^*_{d,k}}}\dbinom{k}{i_1,i_2,\ldots,i_d}~ \prod_{j=1}^d x_j^{i_j}\\
& +k\cdot \sum_{1\leq i<j \leq d} \Big( x_i \cdot x_j^{k-1}+ x_i^{k-1} \cdot x_j\Big)
\end{align*}
by the Multinomial Theorem. From that, we immediately deduce the inequality
\begin{align*}
\frac{1}{1-\sum_{i=1}^d x_i^{k}}\cdot \sum_{1\leq i<j \leq d} \Big(x_i \cdot x_j^{k-1}+ x_i^{k-1} \cdot x_j\Big) \leq k^{-1}\,.
\end{align*}

This shows that the function
\begin{align*}
F_{d,k}(x_1,x_2,\ldots,x_d)=\frac{\sum_{1\leq i<j\leq d} \Big(x_i\cdot x_j^{-1+k}+x_j\cdot x_i^{-1+k} \Big)}{1-\sum_{i=1}^d x_i^k}
\end{align*}
is bounded from above, and so its supremum on the domain defined by $\sum_{i=1}^d x_i=1$ and the inequalities $0<x_1,x_2,\ldots,x_d<1$ exists and is finite:
\begin{align*}
\sup_{\substack{0<x_1,x_2,\ldots,x_d<1\\x_1+x_2+\cdots + x_d=1}} F_{d,k}(x_1,x_2,\ldots,x_d) \leq k^{-1}\,.
\end{align*}

On the other hand, we note that
\begin{align*}
F_{d,k}(\KeepStyleUnderBrace{0,0,\ldots,0}_{(d-2)~0's},\epsilon,1-\epsilon) &=\frac{\epsilon \cdot (1-\epsilon)^{k-1}+ \epsilon^{k-1} \cdot (1-\epsilon)}{1-\epsilon^k - (1-\epsilon)^k}\\
&=\frac{\epsilon \cdot (1-\epsilon)^{k-1}+ \epsilon^{k-1} \cdot (1-\epsilon)}{\epsilon\cdot \Big(\sum_{i=0}^{k-1} (1-\epsilon)^i\Big)-\epsilon^k }
\end{align*}
for every $\epsilon >0$. It follows that
\begin{align*}
\lim_{\epsilon \to 0}F_{d,k}(\KeepStyleUnderBrace{0,0,\ldots,0}_{(d-2)~0's},\epsilon,1-\epsilon) &=\lim_{\epsilon \to 0} \frac{(1-\epsilon)^{k-1}+ \epsilon^{k-2} \cdot (1-\epsilon)}{-\epsilon^{k-1}+ \sum_{i=0}^{k-1} (1-\epsilon)^i}=\frac{1}{k}\,,
\end{align*}
as soon as $k\geq 3$. Hence, we obtain
\begin{align*}
\sup_{\substack{0<x_1,x_2,\ldots,x_d<1\\x_1+x_2+\cdots + x_d=1}} F_{d,k}(x_1,x_2,\ldots,x_d)=\frac{1}{k}\,,
\end{align*}
which is the desired result.
\end{proof}

The following lemma gives us the minimum of the function whose supremum is computed in Lemma~\ref{lemUseful}.

\begin{lemma}\label{LemmaForMinDensityBiCater}
For any positive integers $d\geq 2$ and $k\geq 3$, the function
\begin{align*}
F_{d,k}(x_1,x_2,\ldots,x_d)=\frac{\sum_{1\leq i<j\leq d} \big(x_i\cdot x_j^{-1+k}+x_j\cdot x_i^{-1+k} \big) }{1-\sum_{i=1}^dx_i^k}
\end{align*}
subjected to the constraint $\sum_{i=1}^d x_i=1$, on the domain given by the inequalities $0<x_1,x_2,\ldots,x_d<1$ has its minimum at $\displaystyle x_1=x_2=\cdots =x_d=d^{-1}$, i.e.,
\begin{align*}
F_{d,k}(x_1,x_2,\ldots,x_d)\geq F_{d,k}\big(\underbrace{d^{-1},d^{-1},\ldots,d^{-1}}_{d~\text{terms}}\big)=\frac{d-1}{d^{k-1}-1}
\end{align*}	
for all $0<x_1,x_2,\ldots,x_d<1$ such that $\sum_{i=1}^d x_i=1$.
\end{lemma}

\medskip
Let $A = (a_1,a_2,\ldots, a_n)$ and $B = (b_1, b_2, \ldots, b_n)$ be two vectors of real numbers. Assume $a_1\geq a_2\geq \cdots \geq a_n$ and $b_1\geq b_2\geq \cdots \geq b_n$ in this order. We say that the vector $A$ \textit{majorises} the vector $B$ if 
\begin{align*}
\sum_{i=1}^n a_i= \sum_{i=1}^n b_i\,,
\end{align*}
and for every $k\in \{1,2,\ldots,n-1\}$, 
\begin{align*}
\sum_{i=1}^k a_i \geq  \sum_{i=1}^k b_i\,.
\end{align*}

\begin{theorem}[Muirhead's Inequality]\label{Thm:Muirhead}
Consider a sequence $(x_1,x_2,\ldots, x_n)$ of positive real numbers. If $(a_1,a_2,\ldots, a_n)$ majorises $(b_1, b_2, \ldots, b_n)$ then it holds that
\begin{align*}
\sum_{\pi \in S_n}~\prod_{i=1}^n x_{\pi (i)}^{a_i} \geq \sum_{\pi \in S_n}~\prod_{i=1}^n x_{\pi (i)}^{b_i}\,,
\end{align*}
where the sum is taken over the set $S_n$ of all permutations of $\{1,2,\ldots,n \}$. There is equality if and only if either $a_i=b_i$ for all $i\in\{1,2,\ldots,n\}$, or all the $x_i$'s are equal.
\end{theorem}

A proof of this result can be found, for instance, in the book on inequalities by Hardy, Littlewood and P\'{o}lya~\cite[p.~44-45]{hardy1952inequalities}.

\begin{proof}[Proof of Lemma~\ref{LemmaForMinDensityBiCater}]
Fix $d\geq 2$ and $k\geq 3$. Let $S_d$ be the set of all permutations of the indices $1,2,\ldots,d$. Since $F_{d,k}(x_1,x_2,\ldots,x_d)>0$ by definition, the Multinomial Theorem gives
\begin{align*}
\frac{1}{F_{d,k}(x_1,x_2,\ldots,x_d)}=\frac{\sum_{(i_1,i_2,\ldots,i_d)\in V_{d,k}}\dbinom{k}{i_1,i_2,\ldots,i_d}  \prod_{j=1}^d x_j^{i_j}}{\sum_{1\leq i<j\leq d} \big(x_i\cdot x_j^{-1+k}+x_j\cdot x_i^{-1+k} \big)}\,,
\end{align*}
where $V_{d,k}$ is the set
\begin{align*}
\Big\{(i_1,i_2,\ldots, i_d):~& i_1,i_2,\ldots,i_d~\text{nonnegative integers},\\
& i_1+i_2+\cdots +i_d=k,~\text{and none of them is $k$} \Big\}\,.
\end{align*}

Note that for every $(i_1, i_2, \ldots, i_d) \in V_{d,k}$ such that $i_1\geq i_2\geq \cdots \geq i_d$, the vector  $(k-1,1,\underbrace{0,0,\ldots,0}_{(d-2)~0's})$ majorises $(i_1, i_2, \ldots, i_d)$. Thus, we obtain
\begin{align*}
(d-2)! \cdot \sum_{1\leq i<j\leq d} \Big(x_i\cdot x_j^{-1+k}+x_j\cdot x_i^{-1+k} \Big) \geq \sum_{\pi \in S_d}~\prod_{j=1}^d x_{\pi(j)}^{i_j}
\end{align*}
by Muirhead's Inequality. On the other hand, we also have
\begin{align*}
\frac{d!}{F_{d,k}(x_1,x_2,\ldots,x_d)}=\frac{\sum_{(i_1,i_2,\ldots,i_d)\in V_{d,k}}\dbinom{k}{i_1,i_2,\ldots,i_d}  \sum_{\pi \in S_d}~\prod_{j=1}^d x_{\pi(j)}^{i_j}}{\sum_{1\leq i<j\leq d} \big(x_i\cdot x_j^{-1+k}+x_j\cdot x_i^{-1+k} \big)}\,.
\end{align*}

Therefore, it follows that
\begin{align*}
\frac{d!}{F_{d,k}(x_1,x_2,\ldots,x_d)} \leq (d-2)! \cdot \sum_{(i_1,i_2,\ldots,i_d)\in V_{d,k}}\dbinom{k}{i_1,i_2,\ldots,i_d} = (d-2)! \cdot (d^k -d)\,,
\end{align*}
and hence, we establish that
\begin{align*}
F_{d,k}(x_1,x_2,\ldots,x_d) \geq \frac{d-1}{d^{k-1}-1}=F_{d,k}\big(\underbrace{d^{-1},d^{-1},\ldots,d^{-1}}_{d~\text{terms}}\big)\,.
\end{align*}
Moreover,
\begin{align*}
F_{d,k}\big(\underbrace{d^{-1},d^{-1},\ldots,d^{-1}}_{d~\text{terms}}\big)=\frac{d-1}{d^{k-1}-1}\,.
\end{align*}
This completes the proof.
\end{proof}

We can now give a proof of  Theorem~\ref{minVslimInf}.

\begin{proof}[Proof of Theorem~\ref{minVslimInf}]
Fix $d\geq 2$. First of all, we want to prove that
\begin{align}\label{lim inf binary caterpillars in d ary trees}
\liminf_{\substack{|T|\to \infty \\ T~\text{strictly $d$-ary tree}}}  \gamma(F^2_k,T) =\frac{k!}{2}\cdot  (d-1)^{k-1} \cdot \prod_{j=1}^{k-1}(d^j-1)^{-1}
\end{align}
for every $k\geq 2$. Our approach is an adaptation of~\cite[Proof of Theorem~7]{czabarka2016inducibility}. Setting
\begin{align*}
b_k=\frac{1}{2} \cdot (d-1)^{k-1} \cdot \prod_{j=1}^{k-1}(d^j-1)^{-1}\,,
\end{align*} 
we show that for every positive integer $k\geq 2$, the inequality
\begin{align*}
c(F^2_k,T)\geq b_k\cdot n^k-\frac{1}{(k-1)!} \cdot n^{k-1}
\end{align*}
is satisfied for every strictly $d$-ary tree $T$ with $n$ leaves.

\medskip
The case $k=2$ is essentially obvious as $c(F^2_2,T)=\binom{|T|}{2}$ and $b_2=1/2$ by definition. The proof of the general case goes by induction on $n$. Since $d\geq 2$, it is easy to see that
\begin{align*}
\frac{d^k -1}{d-1}=d^{k-1}+d^{k-2}+\cdots + d +1 \geq k
\end{align*}
for every $k\geq 2$. Thus, we have $b_k \leq 1/(2\cdot (k-1)!) \leq 1/(k-1)!$ meaning that the base case $n=1$ is true. We can then assume that $k\geq 3$ and $n>1$.

For the induction step, consider the $d$ branches $T_1,T_2,\ldots, T_d$ of an arbitrary strictly $d$-ary tree $T$ and suppose that they have $\alpha_1 \cdot n,~\alpha_2 \cdot n, \ldots,~\alpha_d \cdot n$ leaves, respectively. In this setting, by replacing $D$ with $F^2_k$ in equation~\eqref{The general recursion}, we obtain the following formula:
\begin{align}\label{copyinT}
c\big(F^2_k,T \big)=\sum_{i=1}^d c\big(F^2_k,T_i\big)+\sum_{\substack{1\leq i,j\leq d\\ i\neq j}} \alpha_i\cdot n \cdot c\big(F^2_{k-1},T_j\big)\,,
\end{align}
which is valid for every $k\geq 3$. Next, we apply the induction hypothesis: this gives
\begin{align*}
c\big(F^2_k,T\big)&\geq \sum_{i=1}^d \Big(b_k\cdot (\alpha_i \cdot n)^k-\frac{1}{(k-1)!}\cdot (\alpha_i \cdot n)^{k-1} \Big)\\
& +\sum_{\substack{1\leq i,j\leq d\\ i\neq j}} \alpha_i\cdot n \cdot \Big( b_{k-1}\cdot (\alpha_j \cdot n)^{k-1}-\frac{1}{(k-2)!} \cdot (\alpha_j \cdot n)^{k-2}\Big)\\
&= \Bigg( b_k \cdot \sum_{i=1}^d \alpha_i^k  ~ + ~ b_{k-1}\cdot \sum_{\substack{1\leq i,j\leq d\\ i\neq j}} \alpha_i \cdot \alpha_j^{k-1} \Bigg) \cdot n^k\\
& - \frac{1}{(k-2)!} \cdot \Bigg(\frac{1}{k-1}\cdot \sum_{i=1}^d \alpha_i^{k-1} + \sum_{\substack{1\leq i,j\leq d\\ i\neq j}} \alpha_i \cdot \alpha_j^{k-2}\Bigg) \cdot n^{k-1}\,.
\end{align*}

Using the identity
\begin{align*}
b_k=\frac{d-1}{d^{k-1}-1}\cdot b_{k-1}
\end{align*}
along with Lemma~\ref{LemmaForMinDensityBiCater}, we get
\begin{align}\label{Interm}
b_k \leq b_{k-1} \cdot \frac{\sum_{\substack{1\leq i,j\leq d\\ i\neq j}}~\alpha_i \cdot \alpha_j^{k-1}}{1-\sum_{i=1}^d \alpha_i^k}\,,
\end{align} 
and this takes us to the inequality
\begin{align*}
c(F^2_k,T)\geq b_k \cdot n^k - \frac{1}{(k-2)!}\cdot \Bigg(\frac{1}{k-1}\cdot \sum_{i=1}^d \alpha_i^{k-1} + \sum_{\substack{1\leq i,j\leq d\\ i\neq j}} \alpha_i \cdot \alpha_j^{k-2} \Bigg) \cdot n^{k-1}\,.
\end{align*}

On the other hand, we know from Lemma~\ref{lemUseful} that 
\begin{align*}
\frac{\sum_{\substack{1\leq i,j\leq d\\ i\neq j}}~\alpha_i \cdot \alpha_j^{k-2} }{1- \sum_{i=1}^d \alpha_i^{k-1}} \leq \frac{1}{k-1}
\end{align*}
for all $0<\alpha_1,\alpha_2,\ldots,\alpha_d<1$ such that $\sum_{i=1}^d \alpha_i=1$, provided that $k\geq 4$. In this case, we are done immediately. However, for $k=3$, equation~\eqref{copyinT} becomes
\begin{align*}
c\big(F^2_3,T \big)&=\sum_{i=1}^d c\big(F^2_3,T_i\big)+\sum_{\substack{1\leq i,j\leq d\\ i\neq j}} \alpha_i\cdot n \cdot \binom{\alpha_j \cdot n}{2}\\
&\geq \sum_{i=1}^d \Big(b_3\cdot (\alpha_i \cdot n)^3-\frac{1}{2}\cdot (\alpha_i \cdot n)^2 \Big) +\frac{1}{2}\cdot \sum_{\substack{1\leq i,j\leq d\\ i\neq j}} \alpha_i\cdot n \cdot \big((\alpha_j \cdot n)^2- \alpha_j \cdot n \big)\\
&= \Bigg( b_3 \cdot \sum_{i=1}^d \alpha_i^3  ~ + ~ \frac{1}{2} \cdot \sum_{\substack{1\leq i,j\leq d\\ i\neq j}} \alpha_i \cdot \alpha_j^2 \Bigg) \cdot n^3 - \frac{1}{2} \cdot \Bigg(\sum_{i=1}^d \alpha_i^2 + \sum_{\substack{1\leq i,j\leq d\\ i\neq j}} \alpha_i \cdot \alpha_j\Bigg) \cdot n^2\,,
\end{align*}
where the inequality in the second step follows from the induction hypothesis. Therefore, using the identity
\begin{align*}
1- \sum_{i=1}^d \alpha_i^2 =\sum_{\substack{1\leq i,j\leq d\\ i\neq j}}~\alpha_i \cdot \alpha_j\,,
\end{align*}
together with inequality~\eqref{Interm} (as $b_2=1/2$), we deduce that
\begin{align*}
c\big(F^2_3,T \big)\geq  b_3 \cdot n^3 - \frac{1}{2} \cdot n^2\,,
\end{align*}
and this completes the induction proof.

\medskip
Notice that the right side of equation~\eqref{lim inf binary caterpillars in d ary trees} appears already in Corollary~\ref{FkrInCDh}. That is, we have
\begin{align*}
\liminf_{\substack{|T|\to \infty \\ T~\text{strictly $d$-ary tree}}}  \gamma\big(F^2_k,T\big)=\lim_{h\to \infty} \gamma\big(F^2_k,CD^d_h \big)\,,
\end{align*}
and this finishes the proof of the first assertion of Theorem~\ref{minVslimInf}.

\medskip
Let us now tackle the second part of the theorem. The proof is similar to that of Theorem~3 ~in~\cite{AudaceStephanPaper1}.

\medskip
\textbf{Claim}: The sequence 
\begin{align*}
\Bigg(\min_{\substack{|T|=n\\T~\text{$d$-ary tree}}} \gamma\big(F^2_k,T \big)\Bigg)_{n\geq k}
\end{align*}
is nondecreasing, that is, we have
\begin{align*}
\min_{\substack{|T|=n-1\\T~\text{$d$-ary tree}}} \gamma\big(F^2_k,T \big) \leq \min_{\substack{|T|=n\\T~\text{$d$-ary tree}}} \gamma\big(F^2_k,T \big)
\end{align*}
for every $n\geq 1+k$. 

For the proof of the claim, let $T$ be a $d$-ary tree with leaf-set $L(T)$ such that $|T|\geq k$. For $l \in L(T)$, denote by $c_{l}(F^2_k,T)$ the number of subsets of leaves of $T$ that involve $l$ and induce a copy of $F^2_k$.
Thus, every leaf of $T$ is involved in $k \cdot c(F^2_k,T)/|T|$ copies of $F^2_k$ on average, and so there exists a leaf $l_1$ of $T$ for which the inequality
\begin{align}\label{InView}
c_{l_1}(F^2_k,T) \geq \frac{k \cdot c(F^2_k,T)}{|T|}
\end{align}
holds. The number of copies of $F^2_k$ in $T$ not involving the leaf $l_1$ is
\begin{align*}
c(F^2_k,T) - c_{l_1}(F^2_k,T)\leq \Big(1-\frac{k}{|T|} \Big)\cdot c(F^2_k,T)
\end{align*}
by virtue of relation~\eqref{InView}. Call $T^{-}$ the $d$-ary tree that results when the leaf $l_1$ of $T$ is removed and the unique vertex adjacent to $l_1$ (if it has outdegree $2$ in $T$) is suppressed. Thus, since $T$ is an arbitrary $d$-ary tree, we get
\begin{align*}
\min_{\substack{|T^{\prime}|=n-1 \\ T^{\prime}~\text{$d$-ary tree}}}  c(F^2_k,T^{\prime}) \leq  c(F^2_k,T^{-}) \leq \Big(1-\frac{k}{n} \Big) \min_{\substack{|T|=n \\ T~\text{$d$-ary tree}}}  c(F^2_k,T)\,,
\end{align*}
so that dividing both sides of this inequality by $\binom{n-1}{k}$, we obtain
\begin{align*}
\min_{\substack{|T^{\prime}|=n-1\\T^{\prime}~\text{$d$-ary tree}}} \gamma(F^2_k,T^{\prime}) \leq \min_{\substack{|T|=n\\T~\text{$d$-ary tree}}} \gamma(F^2_k,T)
\end{align*}
for every $n\geq 1+k$, showing that the sequence
\begin{align*}
\Bigg( \min_{\substack{|T|=n\\T~\text{$d$-ary tree}}} \gamma(F^2_k,T)\Bigg)_{n\geq k}
\end{align*}
is indeed nondecreasing.

Hence, since this sequence is also bounded from above for every $n\geq k$, one obtains
\begin{align*}
\lim_{n\to \infty}\min_{\substack{|T|=n\\T~\text{$d$-ary tree}}} \gamma\big(F^2_k,T \big)=\liminf_{\substack{|T|\to \infty \\ T~\text{$d$-ary tree}}} \gamma\big(F^2_k,T\big)
\end{align*}
and consequently, we get
\begin{align*}
\min_{\substack{|T|=n\\T~\text{$d$-ary tree}}} \gamma\big(F^2_k,T \big) \leq \liminf_{\substack{|T|\to \infty \\ T~\text{$d$-ary tree}}} \gamma\big(F^2_k,T\big)
\end{align*}
for every $n\geq k$. This completes the entire proof of the theorem.
\end{proof}

The proof of Corollary~\ref{cor:Theo} is now immediate as
\begin{align*}
c\big(F^2_k,CD^d_h\big)=\frac{(d-1)^{k-1}}{2}\cdot d^h \cdot \prod_{i=1}^{k-1} \Bigg(\frac{d^h-d^{i-1}}{d^i-1} \Bigg)
\end{align*}
for all $h\geq 1$ (see Theorem~\ref{CompProport}).

\section{Conclusion}

We conclude this short note with an open question. To formalise the question, we need to define a new class of binary trees. These trees are already considered in previous papers \cite{DossouOloryWagner,czabarka2016inducibility}. A binary tree $T$ is called \textit{even} if for every internal vertex $v$ of $T$, the number of leaves in the two branches of the subtree of $T$ rooted at $v$ differ at most by one.

It is easy to see that there is only one such a binary tree for every given number of leaves. We denote the $n$-leaf even binary tree by $E^2_n$; see Figure~\ref{E2.11} for the tree $E^2_{11}$.

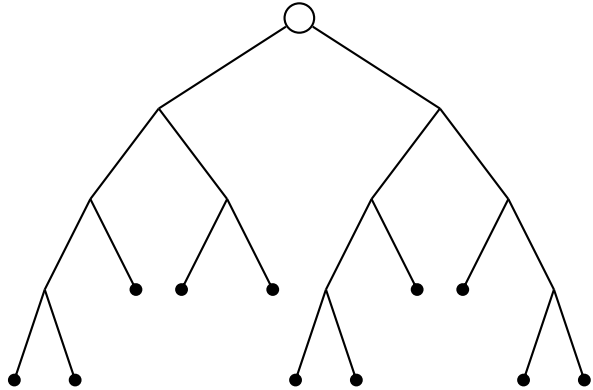
\begin{figure}[htbp]\centering
\begin{tikzpicture}[thick,level distance=12mm]
\tikzstyle{level 1}=[sibling distance=37mm]
\tikzstyle{level 2}=[sibling distance=18mm]
\tikzstyle{level 3}=[sibling distance=12mm]
\tikzstyle{level 4}=[sibling distance=8mm]
\node [circle,draw]{}
child {child {child {child {[fill] circle (2pt)}child {[fill] circle (2pt)}}child {[fill] circle (2pt)}}child {child {[fill] circle (2pt)}child {[fill] circle (2pt)}}}
child {child {child {child {[fill] circle (2pt)}child {[fill] circle (2pt)}}child {[fill] circle (2pt)}}child {child {[fill] circle (2pt)}child {child {[fill] circle (2pt)}child {[fill] circle (2pt)}}}};
\end{tikzpicture}
\caption{The even binary tree $E^2_{11}$ with $11$ leaves.}\label{E2.11}
\end{figure}

\begin{question}
Is it true that for $n\geq k$, the even binary tree $E^2_n$ has the smallest number of copies of the binary caterpillar $F^2_k$ among all binary trees with $n$ leaves? 

We mention that the case $k\leq 3$ is trivial, while calculations show that the case $k\in \{4,5\}$ is also true for values of $n$ up to $100$.
\end{question}

\end{document}